\documentclass[12pt]{amsart}
\usepackage{graphicx,amsmath}

\newtheorem{thm}{Theorem}

\newtheorem{cor}{Corollary}
\newtheorem{lemma}[thm]{Lemma}

\newtheorem{defn}[thm]{Definition}

\newtheorem{rem}{Remark}

\def\C{{\mathbb C}}

\def\D{{\mathbb D}}

\makeatletter
\begin{document}

\title[univalence of certain integral of harmonic mappings]{A new approach for the univalence of certain integral of harmonic mappings}

\author[H. Arbeláez, V. Bravo, R. Hernández, W. Sierra, and O. Venegas ]{Hugo Arbel\'{a}ez \and Victor Bravo \and Rodrigo Hern\'andez \and  Willy Sierra\and Osvaldo Venegas}

\thanks{The second, and third, and fivth authors were partially supported by Fondecyt Grants \#1190756 .
\endgraf  {\sl Key words:} Univalent Mappings, Integral Transformation, Geometric Function Theory.
\endgraf {\sl 2010 AMS Subject Classification}. Primary: 30C45, 44A20;\,
Secondary: 30C55.}
\maketitle

\begin{abstract}
The principal goal of this paper is to extend the classical problem of find the values of $\alpha\in \C$ for which the mappings, either $F_\alpha(z)=\int_0^z(f(\zeta)/\zeta)^\alpha d\zeta$ or $f_\alpha(z)=\int_0^z(f'(\zeta))^\alpha d\zeta$  are univalent, whenever $f$ belongs to some subclasses of univalent mappings in $\D$, but in the case of harmonic mappings, considering the \textit{shear construction} introduced by Clunie and Sheil-Small in \cite{CSS}.
\end{abstract}

\section{Introduction}
Let $f$ be an univalent analytic function defined on the unit disc $\D$. One of the most interesting problem in Geometric Function Theory was the celebrate Bieberbach conjecture (see page 13 in \cite{gk}), which was finally proved by de Brange in \cite{deBranges}. The various attempts to prove this conjecture led to some other interesting problems, one of which was introduced by Royster in \cite{Ro65}, who asked for which values of $\alpha,$ the integral transformations 
\begin{equation}\label{eq def of F_alpha}
F_\alpha(z)=\int_0^z\left(\frac{f(\zeta)}{\zeta}\right)^\alpha d\zeta ,
\end{equation}
or
\begin{equation}\label{eq def of f_alpha}
f_\alpha(z)=\int_0^z(f'(\zeta))^\alpha d\zeta\,,
\end{equation}
are univalent mappings on the unit disc. Some partial results in this direction have been obtained by several authors, in particular Kim and Merkes in \cite{km} proved that $F_\alpha$ is univalent if $|\alpha|<1/4$ and subsequently, J.A. Pfaltzgraff in \cite{Pf75} proved that $f_\alpha$ is univalent if $|\alpha|<1/4.$ However, it is well known that the mapping
\[f(z)=\exp(\mu \log(1-z)),\quad |\mu+1|\leq 1,\quad |\mu-1|\leq 1,\]
is univalent in $\D$ but, for any $\alpha$ such that $|\alpha|>1/3$ and $\alpha\neq 1$, the corresponding $f_\alpha$ is not univalent. With respect to the transformation $F_{\alpha},$ an analogous result was obtained in \cite{km}. In theses paper, the authors proved that for each $\alpha$ with $|\alpha|>1/2$ there is a $\gamma$ such that the function \[f(z)=ze^{\gamma\log(1-z)}=z(1-z)^\gamma\]
is univalent in $\D$, but the associated function $F_\alpha$ is not. Related to this gap, there is a long standing open problem, which is to find the sharp bounds for the values of $\alpha$ such that the corresponding functions are univalent in the unit disc. The problem of this integral operator and its application to different subclasses of univalent mappings, such as convex and starlike functions, among others, has been studied for many authors because of its intrinsic beauty and simplicity. The reader can find an interesting resume of this results in \cite{G}.

The concept of \textit{linear invariant family} (LIF) of holomorphic mappings was introduced by Ch. Pommerenke in \cite{POM1} and has been widely studied in different scenarios, including harmonic mappings in the plane and several complex variables, see for example \cite{Dur,gk}. A family $\mathcal F$ of normalized locally univalent functions of the form
\[f(z)=z+a_2z^2+a_3z^3+\ldots, \qquad z\in\mathbb{D},\]
 is said to be LIF, if for any function $f\in\mathcal F,$ we have $(f\circ \varphi_a-f(a))/((1-|a|^2)f'(a))\in\mathcal F,$ for all automorphism $\varphi_a(z)=(z+a)/(1+\overline a z)$ of $\D.$ One of the most studied families is the class of normalized univalent mappings, named the class $S$. The reader can find some important properties of this class in \cite{Dur83}. At this point, the \textit{order} of a family $\mathcal F$ is the supremum of the modulus of the second Taylor coefficient $|a_2(f)|$ of the mappings $f\in \mathcal F$, for instance the order of $S$ is 2. In connection with the transformation defined in \eqref{eq def of f_alpha}, J.A. Pfaltzgraff showed in \cite{Pf75} the following theorem:
 
\noindent \textbf{Theorem A.} \textit{Let $\mathcal{F}$ be a linear invariant family of order $\beta.$ If $|\alpha|<\frac{1}{2\beta},$ then $f_{\alpha}$ is univalent for all $f\in\mathcal{F}$.}\\

The principal goal of this note, which continues the work in \cite{VBRHOV1}, is to extend the integral transformations defined in equations \eqref{eq def of F_alpha} and \eqref{eq def of f_alpha} to complex harmonic mappings. In the harmonic case, there is more than one way to do this, in the present paper we study some of those options. In fact, it is well known that each planar harmonic mapping $f,$ defined on a simply connected domain $\Omega\subseteq\mathbb{C},$ can be represented as the sum of two functions; one analytic and the other anti-analytic, this is $f=h+\overline g,$ where $h$ and $g$ are analytic functions on $\Omega$. Thus, $f$ is completely determined, except for an additive constant, by $h$ and $g$. However, under certain hypothesis on the range of $f,$ one can find $f$ as the solution of the system given by the equations $h-g=\varphi$ and $g'/h'=\omega,$ in the case when $h$ is a locally univalent mapping. The function $\omega$ is named the second complex dilatation. In this context, $f$ is called the \textit{horizontal shear} of $\varphi,$ see \cite{CSS}. This will be our point of view to extend the transformations $F_\alpha$ and $f_\alpha$ given by the equations \eqref{eq def of F_alpha} and \eqref{eq def of f_alpha} respectively, to complex harmonic mappings. In the last years we have seen several papers about geometric function theory but in the context of harmonic mappings, since the seminal work of Clunie and Sheil-Small \cite{CSS} and the excellent book of P. Duren \cite{Dur}. In this sense, the present paper is part of the same effort and continues the work presented in \cite{VBRHOV1}.

The manuscript is organized as follows. In Section\,\ref{section preliminaries}, we define the principal elements of planar harmonic mappings to be used in our study. Section\,\ref{section F_alpha} is devoted to extend the integral transformation of the first type given by equation \ref{eq def of F_alpha} to harmonic mappings and to find the real values of $\alpha$ for which the corresponding function $F_\alpha$ is,either univalent function or close-to-convex function, considering that $\varphi$ belongs to different classes of univalent functions. Finally, in Section\,\eqref{section f_alpha} we give an alternative extension of the integral transformation of the second type defined in \eqref{eq def of f_alpha}. This extension is different from the one given in \cite{VBRHOV1}.

\section{Preliminaries}\label{section preliminaries}

A complex-valued harmonic function in a simply connected domain $\Omega$ has
a canonical representation $f=h+\overline g$, where $h$ and $g$ are analytic functions in $\Omega$, that is unique up to an additive constant. When $\Omega=\D$ is convenient to choose the additive constant so that $g(0)=0$. The representation $f=h+\overline g$ is therefore unique and is called the canonical representation of $f$. A result of Lewy \cite{Lewy} states that $f$ is locally univalent if and only if its Jacobian $J_f=|h'|^2-|g'|^2$ does not vanish in $\Omega$. Thus, harmonic mappings are either sense-preserving or sense-reversing depending on the conditions $J_f >0$ or $J_f <0$ throughout the domain $\Omega$ where $f$ is locally univalent, respectively. Since $J_f>0$  if and only if $J_{\overline f}<0,$ throughout this work we will consider sense-preserving mappings in $\D.$ In this case the analytic part $h$ is locally univalent in $\D$ since $h'\neq0$, and the second complex dilatation of $f$, $\omega=g'/h'$, is an analytic function in $\D$ with $|\omega|<1$. The reader can observe that the mapping $f$ is completely determined given the complex dilatation $\omega$ and prescribing $h-g=\varphi$. This construction is called \textit{shear construction} and the mapping $f$ is the horizontal shear of $\varphi$. See \cite{CSS}.

Hern\'andez and Mart\'in \cite{RH-MJ}, defined the harmonic Schwarzian and
pre-Schwarzian derivative for sense-preserving harmonic mappings
$f=h+\overline g$. Using this definition, they generalized different results
regarding analytic functions to the harmonic case
\cite{rhmje,HM-QC,RH-MJ,HM-arch}. The pre-Schwarzian derivative of a
sense-preserving harmonic function $f$ is defined by
\begin{equation}\label{Pf}
P_f=\frac{h''}{h'}-
\frac{\bar{\omega}\omega'}{1-|\omega|^2}=\frac{\partial}{\partial
z}\log(J_f).
\end{equation}
It is easy to see that if $f$ is analytic ($f=h$ and
$\omega=0$) then $P_f=h''/h'$ and recover the classical definition of this
operator. In \cite{RH-MJ} the authors proved that for any $a\in\D$,
$P_{f+\overline {af}}=P_f$. In addition, they showed an extension of Becker's criterion of univalence, which is contained in the following theorem:\\

\noindent {\bf Theorem B.} Let $f=h+\overline{g}$  be a sense-preserving
harmonic function in the unit disc $\D$ with dilatation $\omega$. If for all
$z\in \D$ $$
(1-|z|^2)|zP_f(z)|+\frac{|z\omega'(z)|(1-|z|^2)}{1-|\omega(z)|^2} \leq 1,$$
then $f$ is univalent. The constant 1 is sharp.\\

Observe that using the Schwarz-Pick lemma, $\omega:\D\to\D$ satisfies that \begin{equation}\label{omega*}
\|\omega^*\|=\sup_{z\in\D}\dfrac{|\omega'(z)|(1-|z|^2)}{1-|\omega(z)|^2}\leq1.\end{equation}
In addition we consider $\|\omega\|=\sup\{|\omega(z)|:\,z\in\D\}=\|\omega\|_\infty$.
In \cite{VBRHOV1} the authors proved a generalization of Theorem B, this is:\\

\noindent{\bf Theorem C.} Let $f=h+\overline{g}$ be a sense-preserving harmonic function in the unit disc $\D$ with
dilatation $\omega$, $|\omega|<1$, $c\in \C$ such that $|c|\leq 1$ and $c\neq
-1$. If$$
\left|(1-|z|^2)zP_f(z)+c|z|^2\right|+\frac{|z\omega'(z)|(1-|z|^2)}{1-|\omega(z)|^2}\leq
1,\qquad\qquad z\in \D,$$ then $f$ is univalent. 

\section{On the univalence of $F_\alpha$}\label{section F_alpha}

The main goal of this section is to extend the problem of study the univalence of the integral transformation
\begin{equation}\label{F alpha1} 
	\varphi_\alpha(z)=\int_0^z\left(\dfrac{\varphi(\zeta)}{\zeta}\right)^\alpha d\zeta,
\end{equation}
when $\varphi$ is univalent analytic function, to the setting of harmonic mappings in the plane. To this purpose, we will use the shear construction introduced by Clunie and Sheil-Small in \cite{CSS}.

Let $f=h+\overline g$ be a sense-preserving harmonic mapping in $\mathbb{D}$ with the usual normalization $h(0)=g(0)=0$, $h'(0)=1$, dilatation $\omega,$ and defined as the horizontal shear of a holomorphic function $\varphi$. Since $f$ is sense-preserving mapping, we know that $|\omega|<1$ and $\varphi$ is locally univalent in $\D$. However, to our aim, we need that $\varphi(z)\neq 0$ except possibly at $z=0$. Even though this condition is not guaranteed when $f$ is a sense-preserving harmonic mapping, we will show that in several classical subclass of the univalent harmonic functions, it appears in a natural way.
\begin{defn} Let $f:\D\to \C$ be a sense-preserving harmonic mapping with dilatation $\omega$  and suppose that $f$ is the horizontal shear of a holomorphic function $\varphi.$  We define $F_\alpha$ as the horizontal shear of $\varphi_\alpha$ defined by equation (\ref{F alpha1}), with dilatation $\omega_\alpha=\alpha\omega$.
\end{defn}
\begin{rem}
With the notation of the definition, in virtue of the work of Clunie and Sheil-Small \cite{CSS}, $F_\alpha=H+\overline G,$ where $H,G$ are holomorphic functions satisfying $H-G=\varphi_\alpha$ and $G'/H'=\alpha\omega.$ Therefore $F_\alpha$ is a sense-preserving harmonic mapping when $|\alpha|<1$ and $\varphi_\alpha(z)=0$ if and only if $z=0$. For instance, if $\varphi$ is a starlike mapping with respect to the origin in the unit disc and $\alpha\in[0,1],$ then $\varphi_\alpha$ is a convex mapping (see \cite{MW71}), hence $F_\alpha$ is univalent and convex in the horizontal direction. See \cite{CSS}.  
\end{rem}
\subsection{Results} The rest of the section, $f=h+\overline g:\D\to\C$ will be the horizontal shear of the holomorphic maps $\varphi:\D\to\C$ and we adopt the usual normalization, $h(0)=g(0)=0$, $h'(0)=1$, and $g'(0)=0$. We will use $S^*,$ $SHS,$ and $SHCC$ to denote the class of starlike univalent functions, the class of stable harmonic univalent mapping, and the class of stable harmonic close-to-convex, respectively.  $f=h+\overline g$ belongs to $SHS$ or $SHCC$ if and only if for all $\lambda$ in the unit circle $h+\lambda \overline g$ is univalent or close-to-convex, respectively. See \cite{rhmje}.

\begin{thm}\label{B}
 Let $f=h+\overline{g}$ be a sense-preserving harmonic mapping in $\mathbb{D}$ with dilatation $\omega$. If $\varphi\in S$ then $F_\alpha\in SHS$ for all $\alpha$ such that
\begin{equation}\label{prop 1 F alpha}
|\alpha|\leq \dfrac{1}{2(2+\|\omega^\ast\|(1+\|\omega\|))}.
\end{equation}
\end{thm}
\begin{proof} Since that $\varphi\in  S$ we have that $\varphi(z)=0$ just for $z=0$ and so $F_\alpha=H+\overline{G}$ is well-defined. We observe that for any $\lambda$ in $\overline\D$, the function $\Phi_\lambda=H+\lambda G$ satisfies
\[\Phi'_\lambda=H'(1+\lambda \omega_\alpha)=\varphi'_\alpha(1+\lambda\alpha\omega)/(1-\alpha\omega),\]
whence for all $z\in\D,$ \begin{equation}\label{VARPHI} (1-|z|^2)\left|z\frac{\Phi_\lambda''(z)}{\Phi_\lambda'(z)} \right|= (1-|z|^2)|\alpha|\left|z\frac{\varphi'(z)}{\varphi(z)}-1 +\dfrac{\lambda z\omega'(z)}{1+\lambda\alpha\omega(z)}+\dfrac{z\omega'(z)}{1-\alpha\omega(z)}\right|.\end{equation}
Hence, as $\omega$ is a self-map of the unit disc and $|z\varphi'/\varphi-1|\leq 2/(1-|z|),$  which is a consequence of
the fact that $\varphi$ is univalent mapping and using equation (\ref{omega*}), we obtain that 

$$(1-|z|^2)\left|z\frac{\Phi_\lambda''(z)}{\Phi_\lambda'(z)} \right|\leq 2|\alpha|\left(1+|z|+\|\omega^\ast \|(1+\|\omega\|)|z|\right).$$ Then, $\Phi_\lambda$ satisfies the hypothesis of Becker's criterion of univalence for all $\lambda\in\overline\D$ (see \cite{B72}), whenever $\alpha$ satisfies (\ref{prop 1 F alpha}). Therefore $F_\alpha$ belongs to $SHS$.
\end{proof}

If $\varphi\in S^*$ we can proved that $|z\varphi'/\varphi-1|\leq 2|z|/(1-|z|)<2/(1-|z|)$ as in the class $S$. Therefore the hypothesis of Theorem 2 holds.

\begin{lemma}\label{lemma 1 aux } Let $\mathcal F$ be a linear invariant family of holomorphic mappings in $\D$ of order $\beta<\infty$. For all univalent mapping $\varphi\in\mathcal F$ we have
\begin{equation*}
	(1-|z|^2)\left|\dfrac{z\varphi'(z)}{\varphi(z)}\right|\leq 2\beta,
\end{equation*}
for all $z\in\D.$
\end{lemma}
\begin{proof} It is well known that for all $z\in\D,$ $$\dfrac{|\varphi(z)|}{(1-|z|^2)\varphi'(z)}\geq\dfrac{1}{2\beta}\left[1-\left(\dfrac{1-|z|}{1+|z|}\right)^\beta\right],$$
see for example \cite[page 169]{gk}. Since the function $\delta(x)=x(1+x)^\beta/[(1+x)^\beta-(1-x)^\beta]$ is an increasing function in $[0,1]$, we have that $\delta(x)\leq \delta(1)=1$. Therefore $$(1-|z|^2)\left|\dfrac{z\varphi'(z)}{\varphi(z)}\right|\leq 2\beta\delta(|z|)\leq2\beta,$$
for all $z\in\D.$ 
\end{proof}
\begin{cor} Let $f=h+\overline g$ be a sense-preserving harmonic mapping in $\D$ and $\mathcal F$ a linear invariant family of holomorphic mappings of order $\beta$. If $\varphi$ is an univalent function in $\mathcal F$, we have that:
\begin{enumerate}
\item[(i)] if $\|\omega^*\|(1+\|\omega\|)\leq 1$, then $F_\alpha\in SHS$ for all values of $\alpha$ satisfying $$|\alpha|\leq \dfrac{1}{1+2\beta+\|\omega^*\|^2(1+\|\omega\|)^2}.$$

\item[(ii)] If $\|\omega^*\|(1+\|\omega\|)> 1$, then $F_\alpha\in SHS$ for all values of $\alpha$ satisfying $$|\alpha|\leq \dfrac{1}{2\beta+2\|\omega^*\|(1+\|\omega\|)}.$$
\end{enumerate}
\end{cor}

\begin{proof} Using Lemma\,\ref{lemma 1 aux } and equation (\ref{VARPHI}) we get 
$$(1-|z|^2)\left|z\frac{\Phi_\lambda''(z)}{\Phi_\lambda'(z)} \right|\leq |\alpha|\left(2\beta+1-|z|^2+2|z|\|\omega^\ast \|(1+\|\omega\|)\right).$$ The maximum value of the expression in the right side is attained in $|z|=\|\omega^\ast \|(1+\|\omega\|)$. Thus, if this quantity is less or equal than 1 hence $\Phi_\lambda$ satisfies the hypothesis of Becker's criterion of univalence for all $\lambda\in\overline\D$ when $\alpha$ holds the inequality in (i). On the other hand, if $\|\omega^\ast \|(1+\|\omega\|)>1$ the maximum value of the right side expression holds for $|z|=1$, then we have that $\Phi_\lambda$ satisfies the hypothesis of Becker's criterion of univalence for all $\lambda\in\overline\D$ whenever $\alpha$ satisfies the inequality in (ii).
\end{proof}

We observed that in any case, when $|\alpha|\leq1/(2+2\beta)$ it follows that $F_\alpha$ is a stable harmonic univalent mapping, this happens when $\varphi$ is a convex mapping, which implies that $\beta=1$, and $|\alpha|\leq1/4$. 

\begin{thm} Let $f=h+\overline g$ be a sense-preserving harmonic mapping in $\D$ with dilatation $\omega$. If $\varphi$ is a convex mapping and $\alpha\in(-0.303, 0.707),$ then $F_\alpha$ is SHCC.
\end{thm}

\begin{proof} Since $\varphi$ is a convex mapping, we have that $Re\{z\varphi'(z)/\varphi(z)\}\geq1/2$ and $Re\{\varphi(z)/z\}\geq 0$. So for any $\lambda\in\D$, $\alpha\geq0$, and $\Phi_\lambda$ defined as in Theorem\,\ref{B}, satisfies that
\small{\begin{align*}
\int_{\theta_1}^{\theta_2}Re\left\lbrace 1+ z\frac{\varPhi_\lambda''(z)}{\varPhi_\lambda'(z)} \right\rbrace d\theta &=\int_{\theta_1}^{\theta_2}Re\left\lbrace 1+\alpha\left[z\frac{\varphi'(z)}{\varphi(z)}-1 \right]+ \frac{\lambda z\omega_\alpha'(z)}{1+\lambda \omega_\alpha(z)} +\frac{z\omega_\alpha'(z)}{1-\omega_\alpha(z)} \right\rbrace d\theta\\
& = \int_{\theta_1}^{\theta_2}\left[1-\alpha+\alpha Re\left\lbrace z\frac{\varphi'(z)}{\varphi(z)} + \frac{\lambda z\omega'(z)}{1+\lambda \omega_\alpha(z)} +\frac{z\omega'}{1-\omega_\alpha(z)} \right\rbrace \right]d\theta\\ 
&\geq \left( 1-\dfrac{\alpha}{2}\right) (\theta_2-\theta_1)- \textit{Arg} \left\lbrace \frac{1+\lambda\alpha\omega(re^{i\theta_2})}{1+\lambda\alpha\omega(re^{i\theta_1})}\cdot \frac{1-\alpha\omega(re^{i\theta_2})}{1-\alpha\omega(re^{i\theta_1})}\right\rbrace\\
&\geq -4\arcsin(r\alpha)>-4\arcsin(\alpha), 
\end{align*}}
for all $0\leq \theta_2-\theta_1\leq 2\pi.$ Thus, if $\arcsin(\alpha)\leq\pi/4,$ or equivalently $0\leq\alpha\leq \sqrt 2/2\sim 0.707,$ we conclude that $$\int_{\theta_1}^{\theta_2}Re\left\lbrace 1+ z\frac{\varPhi_\lambda''(z)}{\varPhi_\lambda'(z)} \right\rbrace d\theta>-\pi.$$
In the case when $\alpha<0$, we observe that $$\Phi_\lambda'(z)=\varphi'_\alpha(z)\dfrac{1+\lambda\alpha\omega(z)}{1-\alpha\omega(z)}=\left(\dfrac{\varphi(z)}{z}\right)^\alpha\dfrac{1+\lambda\alpha\omega(z)}{1-\alpha\omega(z)}.$$
From the inequality $Re\{\varphi(z)/z\}>0,$ we get $$|\textit{Arg}\,\{\Phi_\lambda'(z)\}|\leq \left|\textit{Arg}\left\lbrace \left(\frac{\varphi(z)}{z}\right)^\alpha\right\rbrace\right|+2\arcsin(r|\alpha|)\leq 2|\alpha|\arcsin(r)+2\arcsin(r|\alpha|),$$
which implies, since $r<1,$ that $|\textit{Arg}\,\{\Phi_\lambda'(z)\}|<|\alpha|\pi+2\arcsin(|\alpha|).$ Considering now the function $y(x)=x\pi+2\arcsin(x),$ $x\geq 0,$ we see that $y(x)\leq \pi/2$ if and only if $x\leq x_0\sim 0.303$, which is the first positive zero of the equation $2y(x)-\pi=0$. We conclude that $\textit{Re}\,\{\Phi_\lambda'(z)\}>0$ for all $\lambda\in\overline\D$ and hence $\Phi_\lambda$ is a close-to-convex mapping in the unit disc. This complete the proof.
\end{proof}

\section{On the univalence of $f_\alpha$}\label{section f_alpha}

In this section we consider the problem of find the values of $\alpha$ for which implies that $\varphi_\alpha$ defined by
 \begin{equation}\label{f alpha}
 \varphi_\alpha(z)=\int_0^z(\varphi'(\zeta))^\alpha d\zeta,
 \end{equation}
 is univalent, but in the context of complex harmonic mappings. In a similar way to what was done in the previous section, the starting point will be the shear construction.

\begin{defn} Let $f=h+\overline g$ be a sense-preserving harmonic mapping, with dilatation $\omega$ and suppose that $f$ is the horizontal shear of a locally univalent holomorphic function $\varphi$. We define the harmonic mapping $f_\alpha$ with dilatation $\omega_\alpha$ by $f_\alpha=H+\overline{G},$ where $\omega_\alpha=\alpha\omega$ and $H,G$ satisfy $H-G=\varphi_\alpha,$ $\varphi_\alpha$ defined by (\ref{f alpha}).  
\end{defn}

Since $f$ is a sense-preserving harmonic mapping, $\varphi$ is a locally univalent holomorphic mapping. Moreover, $|\omega_\alpha|=|\alpha\omega|<|\alpha|\leq1$ for all $\alpha\in[-1,1]$, therefore $f_\alpha$ is a sense-preserving harmonic mapping. A direct calculation shows that the pre-schwarzian derivative of $f_\alpha$, given by (\ref{Pf}), is
\begin{equation}\label{eq10}
P_{f_\alpha}= \alpha\dfrac{\varphi''}{\varphi'}+\dfrac{\alpha\omega'}{1-\alpha\omega}-\dfrac{\omega_\alpha\overline{\omega_\alpha}}{1-|\omega_\alpha|^2} = \alpha\left[\dfrac{\varphi''}{\varphi'}+\omega'\left(\dfrac{1-\overline{\alpha\omega}}{(1-\alpha\omega)(1-|\alpha|^2|\omega|^2)}\right)\right].
\end{equation} This equation will be used through the section.

\subsection{Results} In this section, we consider the harmonic mapping $f=h+\overline g$ as the horizontal shear of a locally univalent holomorphic mapping $\varphi$ defined in the unit disc, and the complex dilatation $\omega$ as a self-map of $\D.$ Similar to the previous section, we assume the normalization $h(0)=g(0)=0$, $h'(0)=1$, and $g'(0)=0$.

\begin{thm} Let $\mathcal F$ be a linear invariant family of harmonic mappings in $\D$ with order $\beta$, and $f=h+\overline g\in \mathcal{F}$.  Then $f_\alpha$ is univalent in $\D$ when $\alpha$ satisfies $$|\alpha|\leq \dfrac{1}{2\beta+(3+\|\omega\|)\|\omega^*\|}.$$
\end{thm}

\begin{proof} Note that, by (\ref{eq10}), for all $z\in\D,$ 
\begin{equation*}
\left| z(1-|z|^2)P_{f_\alpha}+c|z|^2\right|=\left|z(1-|z|^2)\alpha\left[\dfrac{\varphi''}{\varphi'}+\dfrac{\omega'(1-\overline{\alpha\omega)}}{(1-\alpha\omega)(1-|\alpha\omega|^2)}\right]+c|z|^2\right|.
 \end{equation*}
 Since $\varphi=h-g$ then $\dfrac{\varphi''}{\varphi'}=\dfrac{h''}{h'}-\dfrac{\omega'}{1-\omega}$, where $\omega=g'/h'.$ it follows that
\begin{eqnarray*}
  	\left| z(1-|z|^2)P_{f_\alpha}+c|z|^2\right| &\leq &|\alpha|\left| z\left(\dfrac{h''}{h'}(1-|z|^2)-2\overline{z}\right)\right|+|(2\alpha+c)z^2|\\
& & +|\alpha|\left|\dfrac{\omega'}{1-\omega}(1-|z|^2)\right|+|\alpha|\left|\dfrac{ \omega'(1-|z|^2)}{1-|\alpha\omega|^2}\right|\\[0.3cm]
&\leq & |\alpha|(2\beta+(2+\|\omega\|)\|\omega^*\|)+|2\alpha+c|.
\end{eqnarray*} 
Thus, 
\begin{equation*}  \left| z(1-|z|^2)P_{f_\alpha}+c|z|^2\right|+ \frac{|z\omega_\alpha'|(1-|z|^2)}{1-|\omega_\alpha|^2} \leq |\alpha|\left(2\beta+(3+\|\omega\|)\|\omega^*\|\right)+|2\alpha+c|.
\end{equation*}
 Using Theorem C, with $c=-2\alpha\neq-1,$ we obtain that $f_\alpha$ is univalent in $\D,$ if $$|\alpha|\leq \dfrac{1}{2\beta+(3+\|\omega\|)\|\omega^*\|},$$
which ends the proof.
\end{proof}

A domain $\Omega\subset\mathbb{C}$ is {\it $m$-linearly connected} if
there exists a constant $m<\infty$ such that any two points
$w_1,w_2\in\Omega$ are joined by a path $\gamma\subset\Omega$ of
length $l(\gamma)\leq m|w_1-w_2|$, or equivalently (see \cite{P92}),
${\rm diam}(\gamma)\leq m|w_1-w_2|$. Such a domain is necessarily
a Jordan domain, and for piecewise smoothly bounded domains,
linear connectivity is equivalent to the boundary's having no
inward-pointing cusps.

\begin{lemma} Let $f=h+\overline g$ be a sense-preserving harmonic mapping in $\D,$ which is the horizontal shear of univalent mapping $\varphi$, such that $\varphi(\D)$ is a $m$-linearly connected domain of $\mathbb{C}$. Then $f=h+\overline g$ is univalent if $\|\omega\|<1/(2m+1)$.
\end{lemma}
\begin{proof} Let $z_1$ and $z_2$ in the unit disc, such that $f(z_1)=f(z_2)$, which is equivalent to $\varphi(z_1)+\mbox{Re}\{g(z_1)\}=\varphi(z_2)+\mbox{Re}\{g(z_2)\}$. Since $\varphi(\D)$ is $m$-linearly connected, there exists a path $\gamma$ joining $\zeta_1=\varphi(z_2)$ and $\zeta_2=\varphi(z_2)$ such that $\ell_\gamma\leq m|\zeta_1-\zeta_2|$. But $\varphi(z_1)-\varphi(z_2)=-2\mbox{Re}\{g(z_1)-g(z_2)\}$ and $\varphi$ is univalent mapping, then $$|\zeta_1-\zeta_2|=2\left|\mbox{Re}\{g\circ \varphi^{-1}(\zeta_1)-g\circ \varphi^{-1}(\zeta_2)\}\right|\leq2\left|\int_{\zeta_1}^{\zeta_2}(g\circ \varphi^{-1})'(t)dt\right|.$$ However, $$(g\circ \varphi^{-1})'(t)=g'(z)/(h'(z)-g'(z))=\omega(z)/(1-\omega(z)),$$ where $\varphi(z)=t.$ Since in the integral can be used any path joining $\zeta_1$ and $\zeta_2$, we get $$|\zeta_1-\zeta_2|\leq \dfrac{2\|\omega\|}{1-\|\omega\|}\int_{\gamma}|dt|=\dfrac{2\|\omega\|}{1-\|\omega\|}\ell_\gamma<\dfrac{2m\|\omega\|}{1-\|\omega\|}|\zeta_1-\zeta_2|.$$
But for the considered values of $\|\omega\|$ it follows that $2m\|\omega\|/(1-\|\omega\|)<1$, then $|\zeta_1-\zeta_2|=0$. This complete the proof.
\end{proof}
\begin{thm} Let $f=h+\overline g$ be a sense-preserving harmonic mapping, such that $\varphi=h-g$ is a convex mapping. For all $\alpha\in[0,1]$ such that $|\alpha|\|\omega\|<1/3$, the corresponding $f_\alpha$ is univalent in $\D$.
\end{thm}

\begin{proof} Because of $\varphi'_\alpha=(\varphi')^\alpha$ we have
	$$1+\mbox{Re}\left\{z\dfrac{\varphi_\alpha''}{\varphi_\alpha'}(z)\right\}=1+\alpha\left(\mbox{Re}\left\{z\dfrac{\varphi''}{\varphi'}(z)\right\}+1\right)-\alpha\geq 1-\alpha\geq0,$$
	therefore $\varphi_\alpha$ is a convex mapping. Arguing as in the proof of the previous lemma, we conclude that the univalence of $f_\alpha$ depends if $2\|\omega_\alpha\|/(1-\|\omega_\alpha\|)<1$ (in this case $m=1$), which occurs when $\alpha\|\omega\|<1/3$.
\end{proof}


Facultad de Ciencias, Universidad Nacional de Colombia,  sede Medell\'{\i}n, Colombia. Supported by Universidad Nacional de Colombia, Hermes Code 41381. \email{hjarbela@unal.edu.co}. \\ 

Facultad de Ingenier\'ia y Ciencias, Universidad Adolfo Ib\'a\~nez, Av. Padre Hurtado 750, Vi\~na del Mar, Chile. \email{victor.bravo.g@uai.cl}.\\

Facultad de Ingenier\'ia y Ciencias, Universidad Adolfo Ib\'a\~nez, Av. Padre Hurtado 750, Vi\~na del Mar, Chile.\email{rodrigo.hernandez@uai.cl}. \\  

Departamento de Matem\'aticas, Universidad del Cauca, Popay\'{a}n, Colombia. The author wishes to thank the Universidad del Cauca for providing time for this work through research project VRI ID 5069. \email{wsierra@unicauca.edu.co}.\\ 

Departamento de Ciencias Matem\'aticas y F\'{\i}sicas, Facultad de Ingenier\'{\i}a, Universidad Cat\'olica de
Temuco, Chile. \email{ovenegas@uct.cl}.

\end{document}